\documentclass{amsart}
\usepackage{comment}
\usepackage{amsmath,amssymb,amsthm,amsfonts,mathrsfs,mathtools}
\usepackage[frame,cmtip,arrow,matrix,line,graph,curve]{xy}
\usepackage{graphpap,color,paralist}
\usepackage[mathscr]{eucal}
\usepackage{mathabx}
\usepackage[pdftex,colorlinks,backref=page,citecolor=blue]{hyperref}
\usepackage{tikz}
\usepackage{epic,eepic}
\usepackage{yfonts}
\usepackage{enumerate}
\usepackage{ytableau}
\usepackage[alphabetic]{amsrefs}
\usepackage{adjustbox}

\theoremstyle{definition}
\newtheorem{theorem}{Theorem}[section]
\newtheorem{definition}[theorem]{Definition}

\newtheorem{lemma}[theorem]{Lemma}
\newtheorem{proposition}[theorem]{Proposition}
\newtheorem{corollary}[theorem]{Corollary}
\newtheorem*{corollary*}{Corollary}

\newtheorem*{theorem*}{Theorem}

\theoremstyle{remark}
\newtheorem{remark}[theorem]{Remark}
\newtheorem{example}[theorem]{Example}
\def\PP{\mathbb{P}}
\def\RR{\mathbb{R}}
\def\CC{\mathbb{C}}
\def\ZZ{\mathbb{Z}}
\def\QQ{\mathbb{Q}}

\def\kk{\mathbb{K}}

\DeclareMathOperator{\rk}{rk}

\DeclareMathOperator{\CH}{CH}

\DeclareMathOperator{\Spec}{Spec}
\DeclareMathOperator{\trdeg}{trdeg}

\DeclareMathOperator{\vol}{vol}

\makeatletter
\newcommand*{\da@rightarrow}{\mathchar"0\hexnumber@\symAMSa 4B }
\newcommand*{\da@leftarrow}{\mathchar"0\hexnumber@\symAMSa 4C }
\newcommand*{\xdashrightarrow}[2][]{%
  \mathrel{%
    \mathpalette{\da@xarrow{#1}{#2}{}\da@rightarrow{\,}{}}{}%
  }%
}
\newcommand{\xdashleftarrow}[2][]{%
  \mathrel{%
    \mathpalette{\da@xarrow{#1}{#2}\da@leftarrow{}{}{\,}}{}%
  }%
}
\newcommand*{\da@xarrow}[7]{%
  \sbox0{$\ifx#7\scriptstyle\scriptscriptstyle\else\scriptstyle\fi#5#1#6\m@th$}%
  \sbox2{$\ifx#7\scriptstyle\scriptscriptstyle\else\scriptstyle\fi#5#2#6\m@th$}%
  \sbox4{$#7\dabar@\m@th$}%
  \dimen@=\wd0 %
  \ifdim\wd2 >\dimen@
    \dimen@=\wd2 %   
  \fi
  \count@=2 %
  \def\da@bars{\dabar@\dabar@}%
  \@whiledim\count@\wd4<\dimen@\do{%
    \advance\count@\@ne
    \expandafter\def\expandafter\da@bars\expandafter{%
      \da@bars
      \dabar@ 
    }%
  }%  
  \mathrel{#3}%
  \mathrel{%   
    \mathop{\da@bars}\limits
    \ifx\\#1\\%
    \else
      _{\copy0}%
    \fi
    \ifx\\#2\\%
    \else
      ^{\copy2}%
    \fi
  }%   
  \mathrel{#4}%
}
\makeatother

\author{Lukas Grund}
\address{Friedrich-Schiller-Universit\"at Jena, Fakult\"at f\"ur Mathematik und Informatik, Institut für Mathematik, Ernst-Abbe-Platz 2, 07743 Jena, Germany}
\email{lukas.grund@uni-jena.de}
\author{June Huh}
\address{Princeton University and Korea Institute for Advanced Study
}
\email{huh@princeton.edu}

\author{Mateusz Micha\l ek}
\address{
University of Konstanz
}
\email{mateusz.michalek@uni-konstanz.de}
\author{Hendrik S\"uss}
\address{Friedrich-Schiller-Universit\"at Jena, Fakult\"at f\"ur Mathematik und Informatik, Institut für Mathematik, Ernst-Abbe-Platz 2, 07743 Jena, Germany}
\email{hendrik.suess@uni-jena.de}

\author{Botong Wang}
\address{University of Wisconsin–Madison}
\email{wang@math.wisc.edu}

\title{Linear operators preserving volume polynomials}

\begin{document}

\begin{abstract}
Volume polynomials measure the growth of Minkowski sums of convex bodies and of tensor powers of positive line bundles on projective varieties. We show that Aluffi’s covolume polynomials are precisely the polynomial differential operators that preserve volume polynomials, reflecting a duality between homology and cohomology. We then present several applications to matroid theory.
\end{abstract}
\maketitle

\section{Introduction}\label{sec:introduction}

In \cite{Minkowski}, Minkowski made a foundational observation that became a cornerstone of convex geometry: For any collection of convex bodies $C=(C_1,\ldots,C_n)$ in $\RR^d$, the function
\[
f_{C}: \RR^n_{\ge 0} \longrightarrow \RR_{\ge 0}, \qquad (x_1,\ldots,x_n) \longmapsto \frac{1}{d!} \vol(x_1C_1+\dots+x_nC_n)
\]
is a degree $d$ homogeneous polynomial in $x_1,\ldots,x_n$.
This polynomial, called the \emph{volume polynomial} of $C$, is then used to define the \emph{mixed volume} of convex bodies as its normalized coefficients
\[
\textrm{MV}(C_{i_1},\ldots,C_{i_d}) \coloneq \frac{\partial}{\partial x_{i_1}}\cdots \frac{\partial}{\partial x_{i_d}} f_{C}(x_1,\ldots,x_n). 
\]
The mixed volumes satisfy the \emph{Alexandrov--Fenchel inequalities}, of which the classical isoperimetric inequality is a special case. For a comprehensive introduction to the Brunn--Minkowski theory, see \cite{Schneider}.

The analogous volume polynomial in algebraic geometry is defined as follows: Let $D=(D_1,\ldots,D_n)$ be a collection of semiample divisors on a $d$-dimensional projective variety $Y$ over an algebraically closed field $k$.\footnote{Throughout this paper, a variety is by definition reduced and irreducible. A Cartier divisor on a complete variety is \emph{semiample} if some positive multiple moves in a basepoint-free linear system. Undefined terms concerning divisors on varieties are as in \cite{Lazarsfeld1}.} Then the \emph{volume polynomial} of $D$ is 
\[
f_D\coloneq \frac{1}{d!}\int_Y \bigg(\sum_{i=1}^n x_i D_i\bigg)^d,
\]
which is a homogeneous polynomial of degree $d$ in $x_1,\ldots,x_n$. 
When the base field is $\CC$ and each $D_i$ is an ample divisor, the restriction of $f_D$ to the nonnegative orthant measures the volume of $Y$ with respect to the K\"ahler class determined by $x$ \cite[Chapter 0]{GriffithsHarris}. 
A standard construction in toric geometry shows that any volume polynomial of convex bodies arises as a limit of volume polynomials of ample divisors on projective varieties \cite[Section 5.4]{FultonToric}. 
Our goal is to demonstrate that the set of all volume polynomials satisfies several remarkable analytic properties.

\begin{definition}\label{def:volpoly}
Let $k$ be an algebraically closed field.
\begin{enumerate}[(1)]\itemsep 5pt
\item A homogeneous polynomial $f$ is a \emph{realizable volume polynomial over $k$} if $f=\lambda f_D$ for some $\lambda \in \QQ_{\ge 0}$ and a collection of semiample divisors $D$ on a projective variety $Y$ over $k$.
\item A homogeneous polynomial $f$ is a \emph{volume polynomial over $k$} if it is a limit of realizable volume polynomials over $k$.
\end{enumerate}
We write $\mathbb{V}^d_n(\mathbb{Q},k)$ for the set of realizable volume polynomials over $k$ of degree $d$ in $n$ variables, and $\mathbb{V}^d_n(\RR,k)$ for the set of all volume polynomials over $k$ of degree $d$ in $n$ variables.
\end{definition}

We say that a quadratic form is \emph{Lorentzian} if it has nonnegative coefficients and its Hessian has at most one positive eigenvalue. 
By \cite[Theorem 1.8]{HHMWW}, we have
\begin{align*}
\mathbb{V}^2_n(\QQ,k)&=\Big\{\emph{Lorentzian quadratic forms in $n$ variables with rational coefficients}\Big\},\\
\mathbb{V}^2_n(\RR,k)&=\Big\{\emph{Lorentzian quadratic forms in $n$ variables}\Big\}.
\end{align*}
We say that a degree $d$ bivariate form $\sum_{k=0}^d c_k \frac{x_1^{k}}{k!}\frac{x_2^{d-k}}{(d-k)!}$ is \emph{Lorentzian} if $(c_k)$ is a log-concave sequence of nonnegative numbers with no internal zeros. Then, by \cite[Theorem 21]{HuhMilnor}, we have
\begin{align*}
\mathbb{V}^d_2(\QQ,k)&=\Big\{\emph{Lorentzian bivariate forms of degree $d$ with rational coefficients}\Big\},\\
\mathbb{V}^d_2(\RR,k)&=\Big\{\emph{Lorentzian bivariate forms of degree $d$}\Big\}.
\end{align*}
In general, a volume polynomial is a \emph{Lorentzian polynomial} in the sense of \cite[Definition 2.1]{BrandenHuh}. 
It is not known whether the set of volume polynomials over $k$ depends on the choice of $k$. On the other hand, the basis generating polynomial of the Fano matroid is a realizable volume polynomial over $k$ if and only if the characteristic of $k$ is $2$, see Example~\ref{ex:Fano}. 
For a Lorentzian polynomial that is not a volume polynomial over $k$ for any $k$, see \cite[Example~14]{HuhICM}.
The distinction between the notions of realizable volume polynomials and volume polynomials will be important in the application to algebraic matroids in Section~\ref{sec:AlgebraicMatroids}.

Aluffi introduced the dual notion of \emph{covolume polynomials} in \cite[Definition 2.1]{Aluffi}. 
To define covolume polynomials and state their main properties, it will be convenient to work with the dual pair of polynomial rings in countably many variables
\[
\RR[\partial]=\varinjlim_{n}  \RR[\partial_1,\ldots,\partial_n]  \quad \text{and} \quad 
\RR[x]=\varinjlim_{n}  \RR[x_1,\ldots,x_n],
\]
where the polynomial rings are given the topology of direct limits. 
We write $\ZZ^\infty_{\ge 0}=\varinjlim_{n}  \mathbb{Z}^n_{\ge 0}$ for the set of exponent vectors of the monomials in $\RR[x]$, and 
 write $x^{[\alpha]}$ for the normalized monomial $\frac{x^\alpha}{\alpha!}$ in $\RR[x]$.
The polynomial ring $\RR[\partial]$ acts on $\RR[x]$ as differential operators by the usual rule
\[
\partial^\alpha \circ x^{[\beta]}\coloneq \begin{cases} x^{[\beta-\alpha]}& \text{if $\alpha \le \beta$,}\\ \hfill 0\hfill & \text{if otherwise,}\end{cases}
\]
where $\alpha \le \beta$ means that their components satisfy $\alpha_i \le \beta_i$ for all $i$.
For any further conventions for multivariate polynomials, we refer to \cite[Section~2]{BrandenHuh}.
For $\mu \in \ZZ^\infty_{\ge 0}$, we consider 
\[
\RR[\partial]_{\le \mu}\coloneq \textrm{span}(\partial^\alpha)_{\alpha \le \mu} \quad \text{and} \quad
\RR[x]_{\le \mu}\coloneq \textrm{span}(x^\alpha)_{\alpha \le \mu}.
\]
Then $\RR[x]_{\le \mu}$ is an $\RR[\partial]$-submodule of $\RR[x]$ generated by $x^{[\mu]}$, and the linear map 
\[
\RR[\partial]_{\le \mu} \longrightarrow \RR[x]_{\le \mu}, \qquad \partial^\alpha \longmapsto \partial^\alpha \circ x^{[\mu]}=x^{[\mu-\alpha]}
\]
is an isomorphism of finite-dimensional vector spaces.

\begin{definition}
Let $g$ be a homogeneous polynomial in $\RR[\partial]_{\le \mu}$.
\begin{enumerate}[(1)]\itemsep 5pt
\item We say that $g$ is a \emph{realizable covolume polynomial over $k$} if $g(\partial) \circ x^{[\mu]}$ is a realizable volume polynomial over $k$.
\item We say that  $g$ is a \emph{covolume polynomial over $k$} if it is a limit of realizable covolume polynomials over $k$.
\end{enumerate}
\end{definition}

As observed in \cite[Remark 2.2]{Aluffi}, the property of being a (realizable) covolume polynomial over $k$ does not depend on the choice of $\mu$. This follows from the cone construction in Section~\ref{sec:volumecovolume}: If $\sum_\alpha c_\alpha x^{[\alpha]}$ is a realizable volume polynomial over $k$, then $\sum_\alpha c_\alpha x^{[\alpha+\beta]}$ is a realizable volume polynomial over $k$ for any nonnegative $\beta$.

Our first main result is the following characterization of realizable covolume polynomials. This parallels the characterization of dually Lorentzian polynomials in \cite[Theorem 1.2]{RSW23}.

\begin{theorem}\label{thm:covolumecharacterization}
The following conditions are equivalent for any polynomial $g \in \QQ[\partial]$.
\begin{enumerate}[(1)]\itemsep 5pt
\item The polynomial $g$ is a realizable covolume polynomial over $k$.
\item For any realizable volume polynomial $f$ over $k$, the polynomial $g(\partial) \circ f(x)$ is a realizable volume polynomial over $k$.
\end{enumerate}
\end{theorem}

The corresponding characterization of covolume polynomials follows by taking limits.

\begin{corollary}\label{cor:covolumecharacterization}
The following conditions are equivalent for any polynomial $g \in \RR[\partial]$.
\begin{enumerate}[(1)]\itemsep 5pt
\item The polynomial $g$ is a covolume polynomial over $k$.
\item For any volume polynomial $f$ over $k$, the polynomial $g(\partial) \circ f(x)$ is a volume polynomial over $k$.
\end{enumerate}
\end{corollary}

Our next theorem is an immediate consequence of Theorem~\ref{thm:covolumecharacterization}.

\begin{theorem}\label{thm:covolumeproduct}
If $g_1$ and $g_2$ are realizable covolume polynomials over $k$, then  $g_1g_2$ is a realizable covolume polynomial over $k$.
\end{theorem}

By taking limits, we recover the following statement of Aluffi \cite[Corollary 2.14]{Aluffi}.

\begin{corollary}\label{cor:covolumeproduct}
If $g_1$ and $g_2$ are covolume polynomials over $k$, then $g_1g_2$ is a covolume polynomial over $k$.
\end{corollary}

Aluffi shows in \cite[Theorem 2.13]{Aluffi} that any nonnegative linear change of coordinates of a covolume polynomial is also a covolume polynomial, and from this he deduces Corollary~\ref{cor:covolumeproduct}. We strengthen Aluffi's theorem in Theorem~\ref{lem:linmapAluffi} by showing that any nonnegative rational linear change of coordinates of a realizable covolume polynomial is a realizable covolume polynomial.

\begin{remark}[Application to algebraic matroids]
The \emph{support} of a polynomial $f$ in $\RR[x_1,\ldots,x_n]$ is the set of all exponent vectors $\alpha$ such that the monomial $x^\alpha$ appears in $f$ with nonzero coefficient.
By \cite[Theorem 2.25]{BrandenHuh}, the support $J$ of any Lorentzian polynomial is \emph{$\mathrm{M}$-convex} subset of $\ZZ^n_{\ge 0}$: 
\begin{quote}
\emph{For any $i$ and $\alpha,\beta \in J$ whose $i$-th coordinate satisfy $\alpha_i>\beta_i$, there is $j$ satisfying 
\[
\alpha_j<\beta_j \ \ \text{and} \ \ \alpha-e^i+e^j \in J \ \ \text{and} \ \ \beta-e^j+e^i \in J,
\]
where $e^i$ denotes the $i$-th standard basis vector of $\ZZ^n$.}
\end{quote}
The notion of $\mathrm{M}$-convex sets originates in discrete convex analysis \cite[Chapter 4]{Murota}. The condition is equivalent to $J$ being the set of bases of an \emph{integral polymatroid} on $[n]$ in the sense of \cite[Chapter 18]{WelshMatroidTheory}, and to $J$ being the set of lattice points of an \emph{integral generalized permutohedron} in $\RR^n$ in the sense of \cite[Section 4]{Postnikov}.  When $J$ consists of zero-one vectors,  the condition is equivalent to $J$ being the set of bases of a matroid on $[n]$  if we identify a subset of $[n]$ with its indicator vector in $\ZZ^n_{\ge 0}$. 
For a comprehensive treatment of discrete convex analysis and $\mathrm{M}$-convexity, see \cite{Murota}.
 For background specific to  matroids, see \cite{OxleyMatroidTheory}.

Every volume polynomial is Lorentzian \cite[Theorem 4.6]{BrandenHuh}, so its support is the set of bases of an integral polymatroid.  In Proposition~\ref{prop:algebraicPolymatroid}, we prove that polymatroids algebraic over $k$ are precisely those arising as supports of realizable volume polynomials over $k$, thereby answering \cite[Question 5.8]{MultiDegreePositive}.
In Theorem~\ref{thm:MatroidIntersection}, we use this connection to deduce the following statement from
Theorem~\ref{thm:covolumeproduct}: If $M_1$ and $M_2$ are  matroids algebraic over $k$, then the  intersection $M_1 \wedge M_2$ is algebraic over $k$.
 This generalizes Piff's theorem that the truncation of an algebraic matroid is algebraic \cite[Section 11.3, Theorem 2]{WelshMatroidTheory}, and complements the dual result of Welsh that the union of matroids algebraic over $k$ is algebraic over $k$ \cite[Section 11.3, Theorem 4]{WelshMatroidTheory}.
\end{remark}

\begin{remark}[Application to Brunn--Minkowski theory]
Corollary~\ref{cor:covolumecharacterization} can be used to deduce new inequalities for mixed volumes of convex bodies, or more generally, for intersection numbers of nef divisors on a projective variety.
For instance, given a Schubert polynomial $s_w(\partial)$ and a volume polynomial $f_C(x)$, any known inequality for the coefficients of a volume polynomial can be applied to 
$s_w(\partial) \circ f_C(x)$ 
to produce another inequality for the coefficients of $f_C(x)$, since every Schubert polynomial is a covolume polynomial \cite[Theorem 6]{SchurLogconcave}.\footnote{Proposition 2.6 in \cite{ross2025diagonalizations} states that  $s_w(\partial) \circ f_C(x)$ is a volume polynomial for any Schubert polynomial $s_w(\partial)$ and volume polynomial $f_C(x)$. The proof given there relies on the statement that the Schubert degeneracy class of $w$ associated to the sum of globally generated line bundles can be represented by an irreducible subvariety, which is not true in general and unknown when all the line bundles are ample. The main results of  \cite{ross2025diagonalizations} and the previous version of the current paper \cite{GS25} depend on this proposition. Our initial motivation for this study was to close this gap.}
For an overview of known inequalities for the coefficients of the volume polynomial, such as the Khovanskii--Teissier inequality or the reverse Khovanskii--Teissier inequality, see \cite{HuhMichalekWang}. 
\end{remark}

We now formulate a dual characterization of volume polynomials. %that is dual to that of covolume polynomials in Theorem~\ref{thm:covolumecharacterization}. 
The polynomial ring $\RR[x]$ acts on $\RR[\partial]$ as differential operators by the rule
\[
x^\beta \cdot \partial^{[\alpha]}\coloneq \begin{cases} \partial^{[\alpha-\beta]}& \text{if $\alpha \ge \beta,$}\\ \hfill 0\hfill & \text{if otherwise.}\end{cases}
\]
Despite its appearance, the following statement is qualitatively  different from  Theorem~\ref{thm:covolumecharacterization}. See Remark~\ref{ex:nvs2n} for a discussion of this asymmetry.

\begin{theorem}\label{thm:volumecharacterization}
The following conditions are equivalent for any polynomial $f \in \QQ[x]$.
\begin{enumerate}[(1)]\itemsep 5pt
\item The polynomial $f$ is a realizable volume polynomial over $k$.
\item For any realizable covolume polynomial $g$ over $k$, the polynomial $f(x) \cdot g(\partial)$ is a realizable covolume polynomial over $k$.
\end{enumerate}
\end{theorem}

The corresponding characterization of volume polynomials follows by taking limits.

\begin{corollary}\label{cor:volumecharacterization}
The following conditions are equivalent for any polynomial $f \in \RR[x]$.
\begin{enumerate}[(1)]\itemsep 5pt
\item The polynomial $f$ is a volume polynomial over $k$.
\item For any covolume polynomial $g$ over $k$, the polynomial $f(x) \cdot g(\partial)$ is a covolume polynomial over $k$.
\end{enumerate}
\end{corollary}

It follows formally from Theorem~\ref{thm:volumecharacterization} that the product of realizable volume polynomials over $k$ is a realizable volume polynomial over $k$, and that the product of volume polynomials over $k$ is a volume polynomial over $k$. These facts are straightforward to verify directly \cite[Lemma 2.4]{ross2025diagonalizations}, unlike the corresponding statements for covolume polynomials in Theorem~\ref{thm:covolumeproduct} and Corollary~\ref{cor:covolumeproduct}.
It is instructive to compare the statements in the simplest case of bivariate polynomials. In this case, the statement for covolume polynomials is equivalent to the assertion that the convolution of log-concave sequences without internal zeros is a log-concave sequence without internal zeros, and the statement for volume polynomials is equivalent to the assertion that the convolution of ultra-log-concave sequences without internal zeros is an ultra-log-concave sequence without internal zeros. Menon gives a direct argument for the former in \cite{Menon}, and Liggett gives a direct argument for the latter in \cite{Liggett}.

Using Theorem~\ref{thm:covolumecharacterization}, we derive a volume analogue of the symbol theorem for Lorentzian polynomials. The study of symbols of linear operators dates back to G\r{a}rding \cite{Garding} and appears prominently in the work of Borcea and Br\"and\'en on the P\'olya--Schur program \cite{BorceaBranden2,BorceaBranden1}.
Let $x=(x_1,x_2,\ldots)$ and $y=(y_1,y_2,\ldots)$ be two sets of variables, and let $T$ be a linear operator
\[
T: \RR[x]_{\le \mu} \longrightarrow \RR[y]_{\le \nu}.
\]
We suppose that $T$ is \emph{homogeneous}, that is, $\deg\, T(x^\alpha)-\deg\, x^\alpha \in \ZZ$ does not depend on $\alpha \le \mu$.

\begin{definition}
The \emph{symbol} of $T$ is the homogeneous polynomial
\[
\textrm{sym}_T(x,y)=\sum_{0 \le \alpha \le \mu}   T(x^{[\alpha]}) x^{[\mu-\alpha]}.
\]
\end{definition}

The symbol theorem for Lorentzian polynomials states that, if the symbol of $T$ is a Lorentzian polynomial, then $T$ sends Lorentzian polynomials to Lorentzian polynomials  \cite[Theorem 3.2]{BrandenHuh}.
We prove the parallel statement for realizable volume polynomials.

\begin{theorem}\label{thm:symbol}
If the symbol of $T$ is a realizable volume polynomial over $k$, then $T$ sends realizable volume polynomials over $k$ to realizable volume polynomials over $k$.
\end{theorem}

\begin{corollary}
If the symbol of $T$ is a volume polynomial over $k$, then $T$ sends volume polynomials over $k$ to volume polynomials over $k$.
\end{corollary}

We use Theorem~\ref{thm:symbol} to show that many familiar operators from the theory of Lorentzian polynomials preserve realizable volume polynomials over $k$ for any $k$:
\begin{enumerate}[--]\itemsep 5pt
\item The \emph{polarization operator} $\Pi^\uparrow$ preserves realizable volume polynomials over $k$ (Proposition~\ref{prop:polarization}).
\item For any nonnegative rational number $t$, the \emph{interlacing operator} $1+t x_i\partial_j$ preserves realizable volume polynomials over $k$ (Proposition~\ref{prop:interlacing}).
\item For any nonnegative rational number $t$, the \emph{symmetric exclusion process} $\Phi^{1,2}_t$ preserves realizable multiaffine volume polynomials over $k$ (Proposition~\ref{prop:symexpro}).
\end{enumerate}
Theorem~\ref{thm:symbol} also yields, for example, the following generalizations of the corresponding statements from \cite{ross2025diagonalizations} for volume polynomials over $k$:
\begin{enumerate}[--]\itemsep 5pt
\item If $f$ is a realizable volume polynomial over $k$, then the lower truncation $f_{\ge \gamma}$ and the upper truncation $f_{\le \gamma}$ are realizable volume polynomials over $k$ (Corollary~\ref{cor:lowerupper}).
\item If $f$ is a realizable volume polynomial over $k$, then the normalization $N(f)$ is a realizable volume polynomial over $k$ (Proposition~\ref{prop:normisvol}).
\item If $N(f_1)$ and $N(f_2)$ are realizable volume polynomials over $k$, then $N(f_1f_2)$ is a realizable volume polynomial over $k$ (Corollary~\ref{cor:prodofdenomisvol}).
\end{enumerate}
The  statements for volume polynomials follow from taking limits.

\begin{comment}
Motivated by the recent work \cite{HHMWW}, which studies cycle and homology classes of projective varieties presented by irreducible subvarieties, we introduce the notation of $\QQ$-realizable volume and covolume polynomials; see Definitions \ref{def:RealizableVolume} and \ref{def:covolume}. 
All of our results on (co)volume polynomials have analogues in the $\QQ$-realizable setting, which we state and prove. 

In Section \ref{sec:applications}, we present several applications of our main theorems, proving that various polynomials are volume polynomials. 

The notion of a (co)volume polynomial being $\QQ$-realizable is analogous to that of a matroid being algebraic. We make this analogy precise in terms of polymatroids in Proposition \ref{prop:algebraicPolymatroid}. The proposition has two applications. 
\begin{theorem*}[Corollary~\ref{cor:AlgebraicIsChow}]
    Every algebraic polymatroid is a Chow polymatroid. 
\end{theorem*}
This corollary gives a positive answer to \cite[Question 5.8]{MultiDegreePositive}.

\begin{theorem*}[Corollary \ref{thm:MatroidIntersection}]
    Given two algebraic matroids (over the same field $\kk$) $M_1$ and $M_2$ on the ground set $[n]$, if their matroid intersection $M_1\wedge M_2$ has expected rank, i.e., $\rk(M_1\wedge M_2)=\rk(M_1)+\rk(M_2)-n$, then $M_1\wedge M_2$ is also algebraic over $\kk$. 
\end{theorem*}
The corollary is dual to the fact that the matroid union of two algebraic matroids is also algebraic; see \cite[Section 11.3, Theorem 4]{WelshMatroidTheory}. 
\end{comment}

\begin{remark}
In the literature on volume and covolume polynomials, it is customary to work over an algebraically closed field. However, for any field $k$, the  notions of volume and covolume polynomials over $k$ coincide with those over its algebraic closure $\overline{k}$, see Proposition~\ref{prop:ArbitraryField0}. Therefore, all the results above remain valid over any field $k$.
\end{remark}

\subsection*{Acknowledgements}
The authors thank Paolo Aluffi, Petter Br\"and\'en, Matt Larson and Jonathan Monta\~no for their insightful comments. 
Lukas Grund and Hendrik S\"uss acknowledge support from DFG grant 539864509.
June Huh is partially supported by the Oswald Veblen Fund and the Simons Investigator Grant.
 Mateusz Micha\l ek  is partially supported by the Charles Simonyi Endowment. 
 Botong Wang is partially supported by the NSF grant DMS-1926686.

\section{Volume and covolume polynomials}\label{sec:volumecovolume}

Let $f=\sum_\alpha c_\alpha x^{[\alpha]}$ be a nonzero degree $d$ realizable volume polynomial over $k$ in $\QQ[x_1,\ldots,x_n]$.
We use the following construction throughout the paper and refer to it as the \emph{basic construction}. 
By definition, there is a $d$-dimensional projective variety $Y$ over $k$, a collection of semiample divisors $D=(D_1,\ldots,D_n)$ on $Y$, and a positive rational number $\lambda$ such that
\[
f=\lambda f_D, \ \ \text{where} \ \ f_D=\frac{1}{d!}\int_Y \bigg(\sum_{i=1}^n x_i D_i\bigg)^d.
\]
Choose a positive integer $m$ such that $mD_i$ is a basepoint-free divisor for every $i$.
 The collection $mD$ of basepoint-free divisors on $Y$ defines a map to the product of projective spaces
\[
Y \longrightarrow \mathbb{P}^\mu  \coloneq \prod_{i=1}^n \mathbb{P}\mathrm{H}^0(Y,\mathcal{O}_Y(D_i))^\vee.
\]
The Chow group of $\PP^\mu$ is the free abelian group
\[
\CH(\PP^\mu)=\bigoplus_{\alpha \le \mu} \ZZ \cdot [\PP^\alpha]= \bigoplus_{\alpha \le \mu} \ZZ \cdot h^{\mu-\alpha} \cap [\PP^\mu],
\]
where $[\PP^\alpha]$ is the class of the product of linear spaces of the form $\prod_{i=1}^n \mathbb{P}^{\alpha_i}$ and $h_i$ is the pullback of the hyperplane class from the $i$-th factor $\mathbb{P}^{\mu_i}$. 
Let $X$ be the image of $Y$ in $\PP^\mu$. 
By the projection formula, there is a positive rational number $\lambda$ such that
\[
\lambda [X]=\sum_{\alpha \le \mu} c_\alpha [\PP^\alpha]=\sum_{\alpha \le \mu} c_\alpha h^{\mu-\alpha} \cap [\PP^\mu],
\]
where $\lambda$ is determined  by the number $m$ and the degree of $Y \to X$.
The volume polynomial of $X$ with respect to $h$ is $\lambda^{-1}f$.
We make the following observations:
\begin{enumerate}[(1)]\itemsep 5pt
\item For $\nu\in \ZZ^n_{\ge 0}$, the equations satisfied by $X$ in $\PP^\mu$ define a subvariety $X^\nu$ of $\PP^{\mu+\nu}$ such that
\[
\lambda[X^\nu]=\sum_{\alpha \le \mu} c_\alpha [\PP^{\alpha+\nu}]=\sum_{\alpha \le \mu} c_\alpha h^{\mu-\alpha} \cap [\PP^{\mu+\nu}].
\]
Thus $f^{\nu}\coloneq \sum_\alpha c_\alpha x^{[\alpha+\nu]}$ is a realizable volume polynomial over $k$ for any $\nu$. 
The cycles $X$ and $X^\nu$ correspond to the same realizable covolume polynomial
$g(\partial)\coloneq \sum_\alpha c_\alpha \partial^{\mu-\alpha}$.
\item For $\nu\in \ZZ^n_{\ge 0}$, choose a linear embedding $\PP^\mu \to \PP^{\mu+\nu}$. The image $X_\nu$ of $X$ satisfies
\[
\lambda [X_\nu]=\sum_{\alpha \le \mu} c_\alpha [\PP^\alpha]=\sum_{\alpha \le \mu} c_\alpha h^{\mu+\nu-\alpha}\cap [\PP^{\mu+\nu}].
\]
Thus $g^\nu\coloneq \sum_{\alpha} c_\alpha \partial^{\mu+\nu}$ is a realizable covolume polynomial over $k$ for any $\nu$. The cycles $X$ and $X_\nu$ correspond to the same realizable volume polynomial $f(x)=\sum_\alpha c_\alpha x^{[\alpha]}$.
\end{enumerate}

We now prove the main technical lemma of this paper.
Let $X$ be a complete homogeneous variety over $k$, that is, a complete variety with a transitive action of a connected algebraic group $G$.
For a rational point $g$ of $G$ and a morphism $\varphi: Y\to X$, we write $gY \to X$ for the composition $g \circ \varphi$.

\begin{lemma}\label{lem:HomogeneousIntersection}
Let $\varphi: Y \to X$ and $\psi:Z \to X$ be proper morphisms over $k$ from irreducible varieties $Y$ and $Z$ to a complete homogeneous variety $X$. Then, for a general point $g$ of $G(k)$, the irreducible components of the fiber product
        \[
        \xymatrix{
        gY\times_X Z\ar[r]\ar[d]&Z\ar[d]\\
        gY\ar[r] &X
        }
        \]
are algebraically equivalent to each other in $Y \times Z$.
\end{lemma}

When $\varphi$ and $\psi$ are closed immersions, we have $gY \times_X Z \simeq gY \cap Z$. In this case, Lemma~\ref{lem:HomogeneousIntersection} says that, for general $g$ in $G(k)$,  the irreducible components of the intersection are algebraically equivalent to each other in $Y$ and in $Z$, and hence in $X$. See Examples~\ref{ex:disconnected} and ~\ref{ex:singular} for irreducible $Y$ and irreducible $Z$ in $X$ such that the intersection of $Z$ with a general translate of $Y$ is disconnected and has positive dimension.

\begin{proof}
By \cite[Theorem 5.2]{CompleteHomogeneousVarieties}, the complete homogeneous variety $X$ must be of the form $A\times H/P$, where $A$ is an abelian variety, $H$ is a connected affine algebraic group, and $P$ is a parabolic subgroup of $H$. Since any parabolic subgroup is connected \cite[Corollary~6.4.10]{SpringerLAG}, we may assume that $X\simeq G/P$, where $G$ and $P$ are connected algebraic groups.
We consider the fiber squares
\[
\xymatrix{
\mathscr{U}_{Y,Z}\coloneq \{g \varphi(y)=\psi(z)\}  \ar[r]\ar[d]& \ar[d]^{\text{id} \times \phi \times \psi} G \times Y \times Z \ar[r] & Y \times Z \ar[d]^{\phi \times \psi}\\
 \mathscr{U}_{X,X} \coloneq \{gx_1=x_2\} \ar[r] &G \times X \times X \ar[r] & X \times X,
}
\]
where the right horizontal arrows are the projections and the left horizontal arrows are the closed immersions.
By generic flatness and the transitivity of the $G$-action on $X$, we see that the projection
\[
\pi_{X,X}:\mathscr{U}_{X,X} \longrightarrow X \times X
\]
 is a flat morphism whose fibers over closed points are isomorphic to $P$, which is irreducible by our assumption. Flatness is preserved under pullbacks, so the projection
 \[
\pi_{Y,Z}:\mathscr{U}_{Y,Z} \longrightarrow Y \times Z
 \]
shares the same properties.
Since $\pi_{Y,Z}$ is open and surjective, $\pi_{Y,Z}^{-1}(y \times z)$ is irreducible for every closed point $y \times z$, and $Y \times Z$ is irreducible,   it follows that $\mathscr{U}_{Y,Z}$ is irreducible \cite[Lemma 5.8.14]{stacks-project}.

Since $G$ is quasi-projective \cite[Lemma 39.8.7]{stacks-project}, we may take its projective compactification $\overline{G}$.
Let $\mathscr{V}_{Y,Z}$ be the closure of $\mathscr{U}_{Y,Z}$ in $\overline{G}\times Y\times Z$, and consider the projection
\[
p: \mathscr{V}_{Y,Z} \longrightarrow \overline{G}.
\]
Since $\mathscr{V}_{Y,Z}$ is irreducible, 
we may apply the following lemma from \cite[Lemma 2.6]{HHMWW} to the reduced induced scheme of $\mathscr{V}_{Y,Z}$:
\begin{quote}
\emph{Let $f:V \to W$ be a proper surjective morphism between irreducible varieties over $k$.
Then the irreducible components of a general fiber of $f$ are algebraically equivalent to each other in $V$.}
\end{quote}
Therefore, for general $g \in G(k)$, the irreducible components of the fiber 
\[
p^{-1}(g) \simeq gY \times_X Z
\]
are algebraically equivalent to each other in $\mathscr{V}_{Y,Z}$.
Since algebraic equivalence is preserved under proper pushforwards, the irreducible components of $p^{-1}(g)$ are algebraically equivalent in $Y \times Z$.
\end{proof}

The following example is a modification of \cite[Example 29]{KollarLandesman}.

\begin{example}\label{ex:disconnected}
Let $G$ be the general linear group $\mathrm{GL}_5$ and let $X$ be the Grassmannian $\textrm{Gr}(2,5)$. 
For any one-dimensional subspace $F_1$ of $k^5$ and three-dimensional subspace $F_3$ of $k^5$, we consider the Schubert varieties
    \begin{align*}
    &Y(F_1)\coloneq Y =\{V \in \textrm{Gr}(2,5) \, | \, \text{$V$ contains $F_1$}\} \simeq \PP^3, \\
    &S(F_3) \coloneq \{V \in \textrm{Gr}(2,5) \, | \, \text{$V$ is contained in $F_3$}\} \simeq \PP^2.
    \end{align*}
The intersection of $Y(F_1)$ and $S(F_3)$ is either $\PP^1$ or empty, depending on whether $F_1$ is in $F_3$ or not. 
Let $H$ be the intersection of four general hyperplanes in $\textrm{Gr}(3,5) \subseteq \PP^9$. This is an irreducible surface in the parameter space for $F_3$.
We consider the irreducible fourfold
\[
Z \coloneq \bigcup_{F_3 \in H} S(F_3) \subseteq X.
\]
From the classical Schubert calculus, we know that there are precisely two $F_3 \in H$ that contain a given general $F_1$.
Thus, a general translate of $Y$ meets $Z$ in two disjoint copies of $\PP^1$ in $\textrm{Gr}(2,5)$.
Lemma~\ref{lem:HomogeneousIntersection} says that these two copies of $\PP^1$ are algebraically equivalent in $Y$ and in $Z$.
\end{example}

The following example illustrates Lemma~\ref{lem:HomogeneousIntersection} when the intersection of $Z$ with a general translate of $Y$ is not smooth.

\begin{example}\label{ex:singular}
Let $k$ be an algebraically closed field of odd characteristic $p$, and let $\iota$ be  the product of the identity map on $\PP^1$ with the Segre embedding
\[
\iota: \PP^1 \times  \PP^1 \times \PP^1 \longrightarrow \PP^1 \times \PP^3, \ \ (a_0,a_1) \times (b_0,b_1) \times (c_0,c_1) \longmapsto (a_0,a_1) \times (b_0c_0,b_0c_1,b_1c_0,b_1c_1).
\]
We consider a smooth threefold $Y$ and a smooth surface $Z$ in $X \coloneq \PP^1 \times \PP^3$ defined by
\[
Y=\Big\{a_0=0\Big\} \ \ \text{and} \ \ 
Z=\iota\Big\{
a_0b_0^{2p}+a_1b_1^{2p}=0\Big\}. 
\]
The intersection of $Z$ with a general translate of $Y$ in $X$ is the disjoint union of two copies of $\PP^1$, each with multiplicity $p$. The two connected components are algebraically equivalent in $Y$ and in $Z$.
\end{example}

We are ready to show that realizable covolume polynomials are precisely the polynomial differential operators that preserve realizable volume polynomials.

\begin{proof}[Proof of Theorem~\ref{thm:covolumecharacterization}] 
We first show that realizable covolume polynomials preserve realizable volume polynomials. We write
\[
f(x)=\sum_\alpha c_\alpha x^{[\alpha]}, \quad  g(\partial)=\sum_\beta d_\beta \partial^\beta, \quad \text{and} \quad g(\partial) \circ f(x)=\sum_\alpha e_\alpha x^{[\alpha]}.
\]
Suppose $f$ is a realizable volume polynomial over $k$ and $g$ is a realizable covolume polynomial over $k$. By the \emph{basic construction}, we may choose positive rational numbers $\lambda_1$ and $\lambda_2$, a nonnegative integral vector $\mu$ satisfying
$
f \in \QQ[x]_{\le \mu}$ and $g \in \QQ[\partial]_{\le \mu}$,
and subvarieties $Y$ and $Z$ of the product of projective spaces $\PP^ \mu$ satisfying
\[
\lambda_1[Y]=\sum_\alpha c_\alpha [\PP^\alpha] \quad \text{and} \quad \lambda_2[Z]=\sum_\beta d_\beta h^\beta \cap [\PP^\mu].
\]
By \cite[Theorem 2]{Kleiman74}, the intersection of $gY$ and $Z$ is equidimensional and has the expected dimension for a general $g$ in the automorphism group of $\PP^\mu$, and hence 
\[
\lambda_1 \lambda_2 [gY \cap Z]=\sum_\alpha e_\alpha [\PP^\alpha].
\]
By Lemma~\ref{lem:HomogeneousIntersection}, any irreducible component of $V$ of $gY \cap Z$ satisfies $\lambda_1 \lambda_2 \lambda_3 [V]=\sum_\alpha e_\alpha [\PP^\alpha]$ for some positive rational number $\lambda_3$. It follows that  $g(\partial) \circ f(x)$ is a volume polynomial over $k$.

For the converse, choose an integral vector $\mu$ satisfying $g \in \QQ[\partial]_{\le \mu}$. Since  $x^{[\mu]}$ is a realizable volume polynomial over $k$, our hypothesis on $g$ implies that $g(\partial) \circ x^{[\mu]}$ is a realizable volume polynomial over $k$. This means that $g$ is a realizable covolume polynomial over $k$, by definition.
\end{proof}

We collect several consequences of Theorem~\ref{thm:covolumecharacterization}. 

\begin{corollary}\label{cor:Schubert acts on volume}
    If $g$ is a Schubert polynomial  and $f$ is a realizable volume polynomial over $k$, then $g({\partial})\circ f({x})$ is a realizable   volume polynomial over $k$.
\end{corollary}

\begin{proof}
By \cite[Theorem 6]{SchurLogconcave}, a Schubert polynomial is a realizable covolume polynomial over $k$. The conclusion follows from Theorem~\ref{thm:covolumecharacterization}.
\end{proof}

\begin{remark}
The authors of \cite{Equivariant} show more generally that any double Richardson polynomial  is a realizable covolume polynomial over $k$. Thus, these polynomials preserve realizable volume polynomials over $k$. 
\end{remark}

It formally follows from Theorem~\ref{thm:covolumecharacterization} that the product of realizable covolume polynomials over $k$ is a realizable covolume polynomial over $k$. We state here a related property of realizable volume polynomials in terms of the \emph{normalization operator} $N$ defined by
\[
N(x^\alpha)=x^{[\alpha]} \ \ \text{for all $\alpha \in \ZZ^\infty_{\ge 0}$.}
\]
The following statement sharpens \cite[Corollary~2.9]{ross2025diagonalizations}. See \cite[Corollary 3.8]{BrandenHuh} for the corresponding statement for Lorentzian polynomials.

\begin{corollary}
  \label{cor:prodofdenomisvol}
  If $N(f_1)$ and $N(f_2)$ are realizable volume polynomials over $k$, then  $N(f_1f_2)$ is a realizable volume polynomial over $k$.
\end{corollary}

\begin{proof}
Choose $\mu$ such that $f_1,f_2 \in \RR[x]_{\le \mu}$. We define polynomials $g_1,g_2\in \RR[\partial]$ by
\[
f_1=\sum_\alpha c_\alpha x^{\alpha}, \ \  f_2=\sum_\alpha d_\alpha x^{\alpha}, \ \ 
g_1=\sum_\alpha c_\alpha \partial^{\mu-\alpha}, \ \ g_2=\sum_\alpha d_\alpha \partial^{\mu-\alpha}.
\]
Since $N(f_1)$ and $N(f_2)$ are realizable volume polynomials over $k$, $g_1$ and $g_2$ are realizable covolume polynomials over $k$. It follows from Theorem~\ref{thm:covolumeproduct} that
the product $g_1g_2$ is a realizable covolume polynomial over $k$. It follows that
\[
N(f_1f_2)=
g_1g_2 \circ x^{[2\mu]}
\]
is a realizable volume polynomial over $k$.
\end{proof}

The following statement recovers a result of Aluffi that any nonnegative linear change of coordinates preserves covolume polynomials \cite[Theorem~2.13]{Aluffi}. 

\begin{theorem}\label{lem:linmapAluffi}
  If $g(\partial)$ is a realizable covolume polynomial over $k$ and $A$ is a matrix with nonnegative rational entries, 
  then $g(A\delta)$ is a realizable covolume polynomial over $k$. 
\end{theorem}

\begin{proof}
The proof is similar to that of \cite[Theorem~2.13]{Aluffi}.
The use of Debarre's connectedness theorem \cite[Th\'eor\`eme 2.2]{olivier1996theoremes} by Aluffi is replaced by Lemma~\ref{lem:HomogeneousIntersection} to obtain the stronger statement.

Let $\partial=(\partial_1,\ldots,\partial_n)$ and $\delta=(\delta_1,\ldots,\delta_m)$ be the variables viewed as column vectors, and let $A=(a_{ij})$ be an $n \times m$ matrix whose entries are nonnegative rational numbers. 
 Replacing $A$ by its multiple by a sufficiently divisible positive integer, we may assume that $a_{ij}$ are nonnegative integers.
Let $d$ be the degree of $g$. By combining the Veronese map, the Segre map, and the linear embedding, we may construct a map
\[
\varphi: (\PP^d)^m \longrightarrow (\PP^e)^n \ \ \text{such that} \ \ \varphi^*(h_i)=a_{i1}h_1+\cdots+a_{im}h_m \ \ \text{for all $i$},
\]
where $h_i$ is the pullback of the hyperplane class from the $i$-th factor and $e \ge d$ is an integer.

We write $\PP^\mu=(\PP^d)^m$, $\PP^\nu=(\PP^e)^n$, and $g(\partial)=\sum_\alpha c_\alpha \partial^\alpha$.
By increasing $e$ using linear embeddings if necessary, 
we can find a positive rational number $\lambda$ and a subvariety $Y$ of $\PP^\nu$ such that
$\lambda[Y]=\sum_\alpha c_\alpha h^{\nu-\alpha} \cap [\PP^\nu]$.
By Lemma~\ref{lem:HomogeneousIntersection}, for a general $g$ in the automorphism group of $\PP^\nu$, the irreducible components of 
\[
\varphi^{-1}(gY) \simeq gY \times_{\PP^\nu} \PP^\mu
\]
are algebraically equivalent to each other in $\PP^\mu$. Therefore, any one of the irreducible components witnesses the fact that $g(A\delta)$ is a realizable covolume polynomial over $k$.
\end{proof}

It is instructive to compare Theorem~\ref{lem:linmapAluffi} with the corresponding statement for realizable volume polynomials, which is easier.

\begin{proposition}
  \label{prop:linmap}
  If $f(x)$ is a realizable volume polynomial over $k$ and $A$ is a matrix with nonnegative rational entries, 
  then $f(Ay)$ is a realizable volume polynomial over $k$.
\end{proposition}

\begin{proof}
It is enough to prove the statement when the entries of  $A$ are nonnegative integers. 
In this case, the assertion follows from the fact that nonnegative integral linear combinations of semiample divisors are semiample.
\end{proof}

Theorem~\ref{thm:volumecharacterization} strengthens the following corollary.

\begin{corollary}
  \label{cor:prodofvol}
If $f_1$ and $f_2$ are realizable volume polynomials over $k$, then $f_1f_2$ is a realizable volume polynomial over $k$.
\end{corollary}

\begin{proof}
The argument is identical to that for \cite[Lemma 2.4]{ross2025diagonalizations}: If $f_1(x)$ is the realizable volume polynomial arising from semiample divisors on $Y_1$ and $f_2(y)$ is the realizable volume polynomial from semiample divisors on $Y_2$, then $f_1(x)f_2(y)$ is the realizable volume polynomial from semiample divisors on $Y_1 \times Y_2$. By Proposition~\ref{prop:linmap}, we conclude that $f_1(x)f_2(x)$ is a realizable volume polynomial over $k$.
\end{proof}

We conclude this section by examining how our results extend when the ground field $k$ is not necessarily algebraically closed. We extend the definition of realizable volume polynomials over any field $k$ as follows: 
\begin{quote}
\emph{A homogeneous polynomial $f$ is a realizable volume polynomial over $k$ if there exist a nonnegative rational number $\lambda$ and a collection of semiample divisors $D$ on a projective variety $Y$ over $k$ such that  $f=\lambda f_D$, where a variety over $k$ means a reduced and irreducible scheme of finite type over $k$. 
 }
\end{quote}
We show that the resulting class of polynomials depends only on the characteristic of $k$.

\begin{proposition}\label{prop:ArbitraryField0}
The following conditions are equivalent for a polynomial $f$ and $p \ge 0$.
\begin{enumerate}[(1)]\itemsep 5pt
\item The polynomial $f$ is a realizable volume polynomial over some field of characteristic $p$.
\item The polynomial $f$ is a realizable volume polynomial over any field of characteristic $p$.
\end{enumerate}
\end{proposition}

\begin{proof}
Let $\ell/k$ be a field extension. We show that $f$ is a realizable volume polynomial over $\ell$ if and only if it is a realizable volume polynomial over $k$. In the argument below, we implicitly use that
 $[X]$ and $[X_{\text{red}}]$ are equal up to a positive multiple
for any irreducible scheme $X$ in the Chow ring of $X$.

First, we consider the case when $\ell$ is a separable algebraic extension of $k$. 
It is enough to prove the statement when $\ell$ is the separable closure of $k$. 
Let $f$ be a realizable volume polynomial over $k$, and let $X$ be a  subvariety of $\PP^\mu$ over $k$ constructed from $f$ using the \emph{basic construction}.
By \cite[Section 3.2, Exercise 2.10]{Liu}, the Galois group of  $\ell/k$ acts transitively on the irreducible components of $X_{\ell}\coloneq X \times_{\text{Spec}(k)} \text{Spec}(\ell)$ and trivially on the Chow group of $\PP^\mu_{\ell}$. Thus, the realizable volume polynomial determined by  any irreducible component of $X_{\ell}$ is equal to $f$ up to  a positive multiple, and hence $f$ is a realizable volume polynomial over $\ell$.

For the converse direction, let $f$ be a realizable volume polynomial over $\ell$, and let $Y$ be a subvariety of $\PP^\mu_{\ell}$ over $\ell$ constructed from $f$ as above.  
The union of the orbits of $Y$ under the Galois group of $\ell/k$ defines a subvariety $X$ of $\PP^\mu$ over $k$.
Thus, the realizable volume polynomial determined by $Y$  is equal to $f$ up to a positive multiple, and hence $f$ is a realizable volume polynomial over $k$.

Second, we consider the case when $\ell$ is a purely inseparable extension of $k$. In this case, the set of realizable volume polynomials over $\ell$ coincide with that over $k$ because the natural map $\PP^\mu_\ell \to \PP^\mu_k$ is a homeomorphism \cite[Section 3.2, Proposition 2.7]{Liu}.

Third, we consider the case when $\ell$ is a purely transcendental extension of $k$.
Let $f$ be a realizable volume polynomial over $k$, and let $X$ be a subvariety of $\PP^\mu$ over $k$ constructed from $f$ as above. 
Since the polynomial ring $R[x]$ is an integral domain for any integral domain $R$, it follows that $X_\ell$ is reduced and irreducible,  witnessing  that $f$ is a realizable volume polynomial over $\ell$.

For the converse direction,  let $f$ be a realizable volume polynomial over $\ell$, and let $Y$ be a subvariety of $\PP^\mu_{\ell}$ over $\ell$ constructed from $f$ as above.  
Since $Y$ is of finite type over $\ell$, we may suppose that $\ell$ is a finitely generated purely transcendental extension of $k$.
Thus, by induction, we may further suppose that $\ell=k(t)$, where $t$ is a transcendental element over $k$.
Since $\overline{k}(t)$ is an algebraic extension of $k(t)$,
we may use our previous analysis to suppose in addition that $k$ is algebraically closed.

By clearing denominators of the equations for $Y \subseteq \PP_\ell^\mu$ and taking the closure of the generic fiber in the resulting integral model,
we can construct a family $\mathscr{Y} \subseteq \mathbb{A}^1 \times \PP^\mu$, proper and flat over $\mathbb{A}^1=\text{Spec}(k[t])$, such that $\mathscr{Y}$ irreducible and $\mathscr{Y} \to \mathbb{A}^1$ has the generic fiber $Y$.
By the lemma on the irreducible components of a general fiber 
\cite[Lemma 2.6]{HHMWW}, used before in the proof of Lemma~\ref{lem:HomogeneousIntersection}, the irreducible components of a general fiber of $\mathscr{Y} \to \mathbb{A}^1$ are algebraically equivalent to each other in $\mathscr{Y}$.
Let $X$ be any irreducible component of a general fiber of $\mathscr{Y} \to \mathbb{A}^1$.
Since the intersection number is locally constant in flat families \cite[Proposition 23.151]{GortzWedhorn}, 
the realizable volume polynomial determined by $X \subseteq \PP^\mu$ is equal to $f$ up to a positive multiple, and hence
$f$ is a realizable volume polynomial over $k$.
\end{proof}

It follows that the properties of being a volume polynomial, a covolume polynomial, a realizable volume polynomial, or a realizable covolume polynomial over $k$ only depend on the characteristic of $k$. 

\section{The symbol theorem for volume polynomials}

The goal is to prove Theorem~\ref{thm:symbol}.  We use two pairs of dual variables:
\[
x=(x_1,x_2,\ldots) \   \text{and} \  \partial=(\partial_1,\partial_2,\ldots),  \quad y=(y_1,y_2,\ldots) \ \text{and} \   \delta=(\delta_1,\delta_2,\ldots).
\]
As before, we let  $\RR[\partial,\delta]$ act on $\RR[x,y]$ as differential operators by the rule
\[
\partial^\alpha\delta^\beta \circ x^{[\zeta]} y^{[\eta]}=\begin{cases} x^{[\zeta-\alpha]} y^{[\eta-\beta]} & \text{if $\alpha \le \zeta$ and $\beta \le \eta$,}\\ \hfill 0 \hfill & \text{if otherwise.}\end{cases}
\]
We define the coefficients $c_{\alpha,\beta}^T$ of a homogeneous linear operator $T: \RR[x]_{\le \mu} \to \RR[y]_{\le \nu}$ by 
\[
\text{sym}_T=\sum_{\alpha \le \mu} T(x^{[\alpha]}) x^{[\mu-\alpha]}=\sum_{\alpha, \beta} c^T_{\alpha,\beta} \, x^{[\mu-\alpha]} y^{[\beta]}.
\]
The \emph{cosymbol} of $T$ is the homogeneous polynomial in the dual variables\footnote{The authors of \cite{RSW23} use the term for a related polynomial that is different from the cosymbol considered here.}
\[
\text{cosym}_T\coloneq \sum_{\alpha, \beta} c^T_{\alpha,\beta} \, \partial^\alpha \delta^{\nu-\beta} \in \RR[\partial,\delta].
\]

\begin{lemma}\label{lem:cosym}
For any homogeneous linear operator $T: \RR[x]_{\le \mu} \to \RR[y]_{\le \nu}$ and  $f \in \RR[x]_{\le \mu}$, 
\[
\left[\text{cosym}_T \circ f(x)\, y^{[\nu]} \right]_{x=0}=T(f(x)).
\]
\end{lemma}

The left-hand side is the polynomial in $y$ obtained by first applying 
the cosymbol of $T$ to the product  $f(x)\, y^{[\nu]}$, then evaluating the result at $x=0$.
For visual simplicity, the parentheses around the product will be omitted.

\begin{proof}
    Both sides of the equation are linear in $T$. Thus, it is enough to verify the identity for homogeneous linear operators $T$ of the form
    \[
T(x^{[\alpha]})=\begin{cases} y^{[\gamma]}& \text{if $\alpha=\beta$,}\\ \hfill 0\hfill &\text{if $\alpha\neq\beta$.}\end{cases}
    \]
    In this case, the symbol of $T$ is the normalized monomial $x^{[\mu-\beta]}y^{[\gamma]}$, and the cosymbol of $T$ is the monomial $\partial^\beta \delta^{\nu-\gamma}$.
Therefore, for any $f=\sum_\alpha c_\alpha x^{[\alpha]}$, we have
\[
\left[\text{cosym}_T \circ f(x)\, y^{[\nu]}\right]_{x=0}=\left[\partial^\beta \delta^{\nu-\gamma}\circ f(x)\, y^{[\nu]}\right]_{x=0}=c_\beta y^{[\gamma]}=T(f(x)). \qedhere
\]
\end{proof}

\begin{proof}[Proof of Theorem~\ref{thm:symbol}]
If the symbol of $T$ is a realizable volume polynomial over $k$, then
the cosymbol of $T$ is a realizable covolume polynomial over $k$.
If $f(x)$ is a realizable volume polynomial over $k$, so is $f(x)y^{[\nu]}$, by Corollary~\ref{cor:prodofvol}.
It follows from Theorem~\ref{thm:covolumecharacterization} that
\[
\text{cosym}_T \circ f(x) y^{[\nu]}
\]
is a realizable volume polynomial over $k$.
By Proposition~\ref{prop:linmap}, we see that
\[
\left[\text{cosym}_T \circ f(x) y^{[\nu]} \right]_{x=0}
\]
is a realizable volume polynomial over $k$.
By Lemma~\ref{lem:cosym}, this is equal to $T(f(x))$.
\end{proof}

\begin{remark}
Geometrically, one may view $T$ as a map $\varphi_T:\CH(\PP^\mu)_\RR \to \CH(\PP^\nu)_\RR$.
If this map is induced by an irreducible correspondence $\Gamma \subseteq \PP^\mu \times \PP^\nu$ so that
    \[
    \varphi_{T}(\Lambda)=p_{2*}\big(\Gamma\cap p_1^*(\Lambda)\big) \ \ \text{for $\Lambda \in \CH(\PP^\mu)_\RR$,}
    \]
 it preserves the classes of irreducible cycles up to a rational multiple, by Lemma~\ref{lem:HomogeneousIntersection}.
\end{remark}

The first item of the following corollary sharpens a statement from \cite[Corollary 2.10]{ross2025diagonalizations}.

\begin{corollary}\label{cor:lowerupper}
Let $f=\sum_\alpha c_\alpha x^{[\alpha]}$ be a realizable volume polynomial over $k$, and let $\gamma \in \ZZ^\infty_{\ge 0}$.
\begin{enumerate}[(1)]\itemsep 5pt
\item The \emph{lower truncation} $f_{\ge \gamma} \coloneq \sum_{\alpha\ge \gamma} c_\alpha x^{[\alpha]}$ is a realizable volume polynomial over $k$.
\item The \emph{upper truncation} $f_{\le \gamma} \coloneq \sum_{\alpha  \le \gamma} c_\alpha x^{[\alpha]}$ is a realizable volume polynomial over $k$.
\end{enumerate}
\end{corollary}

\begin{proof}
For any $\mu \in \ZZ^\infty_{\ge 0}$, the symbol of the lower truncation operator on $\RR[x]_{\le \mu}$ is the product
\[
\sum_{\gamma \le \alpha \le \mu}  x^{[\mu-\alpha]}y^{[\alpha]}=\prod_{i} \left(\sum_{\gamma_i \le \alpha_i \le \mu_i} x_i^{[\mu_i-\alpha_i]}y_i^{[\alpha_i]} \right).
\]
Each factor is a Lorentzian bivariate form, so it is a realizable volume polynomial over $k$ by \cite[Theorem 21]{HuhMilnor}. Since the product of realizable volume polynomials over $k$ is a realizable volume polynomial over $k$ by Corollary~\ref{cor:prodofvol}, this symbol is a realizable volume polynomial over $k$. We conclude the proof of the first item by Theorem~\ref{thm:symbol}.

On the other hand, the symbol of the upper truncation operator on $\RR[x]_{\le \mu}$ is the product
\[
\sum_{0 \le \alpha \le \gamma}  x^{[\mu-\alpha]} y^{[\alpha]}=\prod_{i} \left(\sum_{0 \le \alpha_i \le \gamma_i} x_i^{[\mu_i-\alpha_i]}y_i^{[\alpha_i]} \right).
\]
Again, each factor is a Lorentzian bivariate form, so we may conclude the proof of the second item in the same way.
\end{proof}

We now use the symbol theorem to show that realizable volume polynomials are precisely the polynomial differential operators that preserve realizable covolume polynomials.

\begin{proof}[Proof of Theorem~\ref{thm:volumecharacterization}]

Let $f$ be a realizable volume polynomial over $k$ in $\RR[x]$. We show that $f(x) \cdot g(\partial)$ is a realizable covolume polynomial over $k$ for any realizable covolume polynomial $g$ in $\RR[\partial]$.
If
$f(x)=\sum_\alpha c_\alpha x^{[\alpha]}$ and $g(\partial)=\sum_\beta d_\beta \partial^\beta$, then 
\[
f(x) \cdot g(\partial)=\sum_{\alpha \le \beta} c_\alpha d_\beta\, \frac{\beta!}{\alpha!} \,\partial^{[\beta-\alpha]}=\sum_{\alpha\le \beta} c_\alpha d_\beta \binom{\beta}{\alpha}\partial^{\beta-\alpha}.
\]
For any fixed $\mu$ satisfying $g \in \RR[\partial]_{\le \mu}$, we consider the polynomials
\[
 \quad g^\vee(x)\coloneq \sum_\beta d_\beta x^{[\mu-\beta]} \ \ \text{and} \ \ \Big(f(x) \cdot g(\partial)\Big)^\vee\coloneq \sum_{\alpha \le \beta} c_\alpha d_\beta \binom{\beta}{\alpha} x^{[\mu-\beta+\alpha]}.
\]
Then $g$ is a realizable covolume polynomial over $k$ if and only if $g^\vee$ is a realizable volume polynomial over $k$, and $f \cdot g$ is a realizable covolume polynomial over $k$ if and only if $(f \cdot g)^\vee$ is a realizable volume polynomial over $k$.
Consider the homogeneous linear operator
\[
T:\RR[x]_{\le \mu} \longrightarrow \RR[y], \qquad g^\vee(x) \longmapsto \Big(f(y) \cdot g(\delta)\Big)^\vee.
\]
We use the following key formula for the symbol of $T$:\footnote{The same formula appears in \cite[Proof of Proposition 5.4]{RSW23}, considered there in the context of dually Lorentzian polynomials.}
\[
\text{sym}_T(x,y)=\frac{1}{\mu!}\Big[f(x)(x+y)^\mu\Big]_{\le (\mu,\mu)}.
\]
Since both sides of the identity are linear in $f$, it is enough to verify the identity when $f$ is a monomial, say $x^\gamma$, and the verification is straightforward in this case:
\begin{align*}
\text{sym}_T(x,y)&=\sum_{\alpha \le \mu} T(x^{[\alpha]}) x^{[\mu-\alpha]}=\sum_{\alpha \le \mu-\gamma} \frac{(\mu-\alpha)!}{(\mu-\alpha-\gamma)!} x^{[\mu-\alpha]} y^{[\alpha+\gamma]}
=\frac{1}{\mu!} \Big[x^\gamma (x+y)^\mu\Big]_{\le (\mu,\mu)}.
\end{align*}
 The symbol is a realizable volume polynomial over $k$ because
the product of realizable volume polynomials over $k$ is a realizable volume polynomial over $k$ (Corollary~\ref{cor:prodofvol}) and
the upper truncation of a realizable volume polynomial over $k$ is a realizable volume polynomial over $k$ (Corollary~\ref{cor:lowerupper}).
It follows  from Theorem~\ref{thm:symbol} that $T$ preserves realizable volume polynomials over $k$. Thus, if  $g$ is a realizable covolume polynomial over $k$, then  $f \cdot g$ is a realizable covolume polynomial over $k$.

We now prove the converse direction. 
Suppose $f \cdot -$ preserves realizable covolume polynomials over $k$.
We choose $\mu$ so that $f \in \RR[x]_{\le \mu}$,
and consider the polynomial
\[
\sum_{\beta \le \mu} \partial^\beta \delta^{\mu-\beta} =\prod_i \left( \sum_{\beta_i \le \mu_i} \partial_i^{\beta_i} \delta_i^{\mu_i-\beta_i} \right)  \in \RR[\partial,\delta].
\]
Since its factors are bivariate Lorentzian forms, the product is a realizable covolume polynomial over $k$.
If $f=\sum_\alpha c_\alpha x^{[\alpha]}$, then
\[
 f(x) \cdot  \left(\sum_{\beta \le \mu} \partial^\beta \delta^{\mu-\beta}\right)=\sum_{\alpha \le \beta \le \mu} c_\alpha \frac{\beta!}{\alpha!} \partial^{[\beta-\alpha]}\delta^{\mu-\beta},
\]
and this is a realizable covolume polynomial over $k$ by the assumption on $f \cdot -$. It follows from Theorem~\ref{lem:linmapAluffi} that
\[
 \left[f(x) \cdot  \left(\sum_{\beta \le \mu} \partial^\beta \delta^{\mu-\beta}\right)\right]_{\partial=0}=\sum_\alpha c_\alpha \delta^{\mu-\alpha}=f^\vee(\delta)
\]
is a realizable covolume polynomial over $k$ as well. Therefore, 
$f(x)=f^\vee (\partial)\circ x^{[\mu]}$ is a realizable volume polynomial over $k$.
\end{proof}

\begin{remark}\label{ex:nvs2n}
The proof of Theorem~\ref{thm:covolumecharacterization} gives the following characterization of covolume polynomials over $k$ in $n$ variables:
\begin{quote}
\emph{
A homogeneous polynomial $g \in \RR[\partial_1,\ldots,\partial_n]$ is a covolume polynomial if and only if $g \circ f$ is a volume polynomial for every volume polynomial $f \in \RR[x_1,\ldots,x_n]$.
}
\end{quote}
On the other hand, the proof of Theorem~\ref{thm:volumecharacterization} gives the following characterization of volume polynomials over $k$ in $n$ variables:
\begin{quote}
\emph{
A homogeneous polynomial $f \in \RR[x_1,\ldots,x_n]$ is a volume polynomial if and only if $f \cdot g$ is a covolume polynomial for every covolume polynomial $g \in \RR[\partial_1,\ldots,\partial_{2n}]$.}
\end{quote}
The ``if'' direction of the latter statement fails if we only test against  $g \in \RR[\partial_1,\ldots,\partial_{n}]$. To see this in the simplest setting,  consider the bivariate polynomial with nonnegative coefficients
\[
f=ax_1^3x_2+bx_1^2x_2^2+cx_1x_2^3=6ax_1^{[3]}x_2+4bx_1^{[2]}x_2^{[2]}+6cx_1x_2^{[3]}.
\]
By \cite[Theorem 21]{HuhMilnor}, $f$ is a volume polynomial if and only if $(4b)^2 \ge (6a)(6c)$, and
\[
g=\sum_{k=1}^d d_k \partial_1^k \partial_2^{d-k}=\sum_{k=1}^d d_k k! (d-k)! \partial_1^{[k]} \partial_2^{[d-k]}
\]
is a covolume polynomial if and only if $(d_k)$ is a nonnegative log-concave sequence with no internal zeros.
A direct but lengthy computation shows that 
$f \cdot g$ is a covolume polynomial for every covolume polynomial $g \in \RR[\partial_1,\partial_2]$ if and only if $b^2 \ge ac$, which is weaker than $(4b)^2 \ge (6a)(6c)$.
\end{remark}

\begin{remark}
The proof of Theorem~\ref{thm:volumecharacterization} yields the following strengthening of \cite[Proposition 5.4]{RSW23}:
\begin{quote}
\emph{
A homogeneous polynomial $f \in \RR[x]$ is a Lorentzian polynomial if and only if $f \cdot g$ is a dually Lorentzian  polynomial for every dually Lorentzian polynomial $g \in \RR[\partial]$.}
\end{quote}
\end{remark}

\section{Linear operators preserving volume polynomials}\label{sec:applications}

We use the symbol theorem for volume polynomials to show that various operators from the theory of Lorentzian polynomials in \cite{BrandenHuh} preserve volume polynomials as well.

For fixed $\mu \in \ZZ^\infty_{\ge 0}$ and $d \in \ZZ_{\ge 0}$, we consider the \emph{polarization operator} 
\[
\Pi^\uparrow: \RR[x,z_0]_{\le (\mu,d)} \longrightarrow \RR[y,z_1,\ldots,z_d], \qquad z_0^k x^\alpha  \longmapsto \frac{e_k(z_1,\ldots,z_d)}{\binom{d}{k}} y^\alpha,
\]
where $e_k(z_1,\ldots,z_d)$ is the $k$-th elementary symmetric polynomial in $(z_1,\ldots,z_d)$.
Recall that a polynomial is said to be \emph{multiaffine} if each variable appears with degree at most one. 
Any polynomial can be made multiaffine by successive applications of the polarization operator.
The following statement parallels that for Lorentzian polynomials in \cite[Proposition 3.1]{BrandenHuh}.

\begin{proposition}\label{prop:polarization}
A polynomial $f$ is a realizable volume polynomial over $k$ if and only if $\Pi^\uparrow(f)$ is a realizable volume polynomial over $k$.
\end{proposition}

\begin{proof}
Note that $f$ can be recovered from  $\Pi^\uparrow(f)$ by setting $z_0=z_1=\cdots=z_d$. Therefore,
if $\Pi^\uparrow(f)$ is a realizable volume polynomial, then $f$ is a realizable volume polynomial by Proposition~\ref{prop:linmap}.

For the other implication, we compute the symbol of the polarization operator:
\[
\text{sym}_{\Pi^\uparrow}=\sum_{(\alpha,k) \le (\mu,d)} \Pi^\uparrow\Big(z_0^{[k]} x^{[\alpha]}\Big)\, z_0^{[d-k]} x^{[\mu-\alpha]} 
=\frac{1}{d!} \frac{1}{\mu!} (x+y)^\mu \prod_{i=1}^d (z_0+z_i).
\]
Since each factor is a realizable volume polynomial over $k$, the product is a realizable volume polynomial over $k$. 
The conclusion follows from the symbol theorem for volume polynomials (Theorem~\ref{thm:symbol}).
\end{proof}

The next result improves {\cite[Corollary~3.3]{ross2025diagonalizations}}.
See \cite[Corollary 3.7]{BrandenHuh} for the corresponding statement for Lorentzian polynomials.

\begin{proposition}
  \label{prop:normisvol}
  If $f$ is a realizable volume polynomial over $k$, then the normalization $N(f)$ is a realizable volume polynomial over $k$ 
\end{proposition}

\begin{proof}
For any $\mu \in \ZZ^\infty_{\ge 0}$, the symbol of the normalization operator on $\RR[x]_{\le \mu}$ is the product
\[
\sum_{0\le \alpha \le \mu} \frac{1}{\alpha!} x^{[\mu-\alpha]} y^{[\alpha]}=\prod_{i} \left( \sum_{0\le \alpha_i \le \mu_i} \frac{1}{\alpha_i!} x_i^{[\mu_i-\alpha_i]} y_i^{[\alpha_i]} \right).
\]
Since $\frac{1}{a!}$ is a log-concave sequence in $a$, the conclusion follows from \cite[Theorem 21]{HuhMilnor}, combined with Corollary~\ref{cor:prodofvol} and Theorem~\ref{thm:symbol} as above.
\end{proof}

For  any $t \in \RR_{\ge 0}$ and  $\mu \in \ZZ^\infty_{\ge 0}$, we consider the \emph{interlacing operator}
\[
T_{12}(t):\RR[x]_{\le \mu} \longrightarrow \RR[y], \qquad f(x) \longmapsto (1+t y_1\delta_2)f(y).
\]
The following statement parallels that for Lorentzian polynomials in \cite[Proposition 2.7]{BrandenHuh}.

\begin{proposition}\label{prop:interlacing}
A polynomial $f$ is a realizable volume polynomial over $k$ if and only if $(1+tx_1\partial_2)f$ is a realizable volume polynomial over $k$ for all nonnegative rational number $t$.
\end{proposition}

\begin{proof}
One direction is obtained by setting $t=0$. 
For the other direction, we compute the symbol of the interlacing operator:
\[
\text{sym}_{T_{12}(t)}=\frac{1}{\mu!} \Big((x+y)^\mu+ t y_1 \delta_2 (x+y)^\mu\Big)=\frac{1}{\mu!}(x+y)^{\mu-e^2} \Big(x_2+y_2+t\mu_2 y_1\Big),
\]
where, as before, $e^i$ stands for the $i$-th standard basis vector of $\ZZ^\infty$. 
Since each factor is a realizable volume polynomial over $k$, the product is a realizable volume polynomial over $k$. 
\end{proof}

The following operator $\Phi_{12}(t)$ occurs in the context of the \emph{symmetric exclusion process}:
\[
\Phi_{12}(t):\RR[x] \longrightarrow \RR[y], \quad f \longmapsto t f(y_1,y_2,y_3,\ldots,y_n)+(1-t)f(y_2,y_1,y_3,\ldots,y_n).
\]
It plays a key role in Br\"and\'en's proof of the fact that the space of Lorentzian polynomials $L^d_n$ modulo $\RR_{>0}$ is homeomorphic to a compact ball \cite{BrandenBall}. The following statement parallels that for Lorentzian polynomials in \cite[Proposition 3.9]{BrandenHuh}.

\begin{proposition}\label{prop:symexpro}
  For every rational number $0 \leq t \leq 1$, the linear operator 
  $\Phi_{12}(t)$ preserves multiaffine realizable volume polynomials over $k$.
\end{proposition}

As in the case of Lorentzian polynomials, the multiaffine hypothesis is necessary. For example, $x_1^2$ is a volume polynomial, but $\frac{1}{2}(x_1^2+x_2^2)$ is not a volume polynomial.

\begin{proof}
The symbol of the linear operator $\Phi_{12}(t)$ on $\RR[x_1,\ldots,x_n]_{\le(1,\ldots,1)}$ is
\[
\text{sym}_{\Phi_{12}(t)}=    \Big(t (x_1+y_1)(x_2+y_2) + (1-t)(x_2+y_1)(x_1+y_2)\Big)\prod_{i=3}^n(x_i+y_i).
\]
It suffices to show that the first factor is a realizable volume polynomial over $k$. Its Hessian is
\[
\scalebox{0.8}{$
\begin{pmatrix}
0      & 1      & 1 - t  & t      \\
1      & 0      & t      & 1 - t  \\
1 - t  & t      & 0      & 1      \\
t      & 1 - t  & 1      & 0
\end{pmatrix},$}
\]
which has eigenvalues $2,0,2(t-1),-2t$. Thus, the Hessian is Lorentzian for every $0 \le t \le 1$. By \cite[Theorem 1.8]{HHMWW}, every quadratic Lorentzian polynomial with rational coefficients is a realizable volume polynomial over $k$, and this finishes the proof.
\end{proof}

As in Section~\ref{sec:introduction}, 
the versions of Propositions~\ref{prop:polarization}, ~\ref{prop:normisvol}, ~\ref{prop:interlacing}, and ~\ref{prop:symexpro} 
 for volume polynomials follow by taking limits.

\section{Algebraic matroids and volume polynomials}\label{sec:AlgebraicMatroids}

Recall that the \emph{support} of a polynomial $f$ in $\RR[x]$ is the set of all exponent vectors $\alpha$ such that the monomial $x^\alpha$ appears in $f$ with nonzero coefficient.
We discuss relations between the support of realizable volume polynomials and polymatroids. 
Our statements are valid for any field $k$.
For a reduction to the case of algebraically closed fields, see Proposition~\ref{prop:ArbitraryField0}.

We first recall the definition and  basic properties of polymatroids \cite[Chapter 18]{WelshMatroidTheory}. 
Let $E$ be a finite set. For any subset $A \subseteq E$ and any vector $\alpha \in \ZZ^E_{\ge 0}$, we set $\alpha_A \coloneq \sum_{i \in A} \alpha_i$.

\begin{definition}\label{def:polymatroid}
A function $h:2^E \to \ZZ_{\ge 0}$ is a \emph{polymatroid rank function} if it satisfies the following properties:
\begin{enumerate}[(1)]\itemsep 5pt
  \item \emph{Normalization:} $h(\varnothing) = 0$.
  \item \emph{Monotonicity:} 
    $h(A) \le h(B)$ for all $A \subseteq B \subseteq E$.
  \item \emph{Submodularity:} 
    $h(A\cup B) + h(A\cap B) \le h(A) + h(B)$ for all $A,B \subseteq E$.
\end{enumerate}
A polymatroid rank function $h$ is a \emph{matroid rank function} if $h(A) \le |A|$ for all $A\subseteq E$.
\end{definition}

\begin{definition}
A subset $J$ of $\ZZ^E_{\ge 0}$ is
\emph{$\mathrm{M}$-convex} if it satisfies the \emph{symmetric basis-exchange property}:
For any $i \in E$ and $\alpha,\beta \in J$ whose $i$-th coordinate satisfy $\alpha_i>\beta_i$, there is $j \in E$ satisfying 
\[
\alpha_j<\beta_j \ \ \text{and} \ \ \alpha-e^i+e^j \in J \ \ \text{and} \ \ \beta-e^j+e^i \in J,
\]
where $e^i$ denotes the $i$-th standard basis vector of $\ZZ^E$.
\end{definition}

\begin{comment}
\begin{definition}\label{def:polymatroid}
    A \emph{polymatroid} $\mathscr{P}$ on $E$ is given by a \emph{rank function} $\rk=\rk_\mathscr{P}: 2^E \to \ZZ_{\geq 0}$ that satisfies the following properties:
\begin{enumerate}[(1)]\itemsep 5pt
  \item \emph{Normalization:} $\rk(\varnothing) = 0$.
  \item \emph{Monotonicity:} 
    $\rk(A) \le \rk(B)$ for all $A \subseteq B \subseteq E$.
  \item \emph{Submodularity:} 
    $ \rk(A\cup B) + \rk(A\cap B) \le \rk(A) + \rk(B)$ for all $A,B \subseteq E$.
\end{enumerate}
%The integer $d=\rk(E)$ is called the \emph{rank} of $\mathscr{P}$. 
The polymatroid $\mathscr{P}$ is a \emph{matroid} if $\rk(i) \le 1$ for all $i\in E$.
\end{definition}
\end{comment}

We recall the standard bijection between polymatroid rank functions on $E$ and nonempty $\mathrm{M}$-convex subsets of $\ZZ^E_{\ge 0}$ from \cite[Chapter 4]{Murota}. Firstly, a polymatroid rank function $h$ gives 
\[
J_{h}
\coloneq
\left\{
\alpha \in \mathbb{Z}_{\ge0}^E
\;\middle|\;
\alpha_E= h(E) \ \ \text{and} \ \  
\alpha_A \le h(A)\ \text{for all $A\subseteq E$}
\right\}.
\]
Then $J_h$ is an $\mathrm{M}$-convex subset of $\ZZ^E_{\ge 0}$.
Conversely, an $\mathrm{M}$-convex subset $J$ of $\ZZ^E_{\ge 0}$ defines 
\[
h_J:2^E \longrightarrow \ZZ_{\ge 0}, \quad h_J(A)\coloneq \max\big\{ \beta_A  \mid \text{$\beta\le \alpha$ for some $\alpha \in J$}\big\}.
\]
Then $h_J$ is a polymatroid rank function on $E$.
The constructions $J_h$ and $h_J$ are mutually inverse, providing a polymatroid generalization of the classical cryptomorphism between the  matroid rank function axioms and the symmetric basis-exchange property.
A \emph{polymatroid} $\mathscr{P}$ is a pair $(h=h_J,J=J_h)$, where
 $h$ is the \emph{rank function} of $\mathscr{P}$ and $J$ is the \emph{set of bases} of $\mathscr{P}$.\footnote{What we call a polymatroid is termed an integral polymatroid in \cite[Chapter 18]{WelshMatroidTheory} and \cite[Chapter 12]{MonomialIdeals}. This paper only considers integral polymatroids. In \cite[Chapter 4]{Murota}, the bijection between $h$ and $J$ is formulated more generally for integral submodular functions on $E$ satisfying the normalization condition and nonempty $\mathrm{M}$-convex subsets of $\ZZ^E$.}
A polymatroid $\mathscr{P}$ is a \emph{matroid} if $h$ is a matroid rank function, or equivalently if $J$ consists of zero-one vectors.

An \emph{algebraic matroid over $k$} is a matroid that arises from algebraic independence relations over a field $k$: Given a field extension $F$ of $k$ and a finite subset $E$ of $F$,
 the corresponding algebraic matroid on $E$ has as its independent sets those subsets of $E$ that are algebraically independent over $k$. For basic results on this classical subject, see \cite[Chapter 11]{WelshMatroidTheory} and \cite[Section 6.7]{OxleyMatroidTheory}.

The following definition from \cite{MultiDegreePositive} extends that of algebraic matroids.
For field extensions $F_i \subseteq F$  for $i \in E$, we write $F_A$
for the \emph{composite} of $\{F_i\}_{i \in A}$ in $F$, that is, the smallest subfield of $F$ containing $F_i$ for all $i\in A$.

\begin{definition}\label{def:AlgebraicPolymatroid}
Let $h$ be the rank function of a polymatroid $\mathscr{P}$ on $E$. 
We say that $\mathscr{P}$ is  \emph{algebraic over $k$} if there are field extensions $k \subseteq F_i \subseteq F$  for $i \in E$
 such that 
 \[
h(A)= \trdeg_k (F_A) \ \ \text{for all $A \subseteq E$}.
 \]
\end{definition}

We characterize algebraic polymatroids in terms of their sets of bases.
This gives an affirmative answer to \cite[Question 5.8]{MultiDegreePositive}.

\begin{proposition}\label{prop:algebraicPolymatroid}
 The following statements are equivalent for any polymatroid $\mathscr{P}=(h,J)$ and any field $k$.
    \begin{enumerate}[(1)]\itemsep 5pt
        \item There is a realizable volume polynomial over $k$ whose support is $J$. 
        \item The polymatroid  $\mathscr{P}$ is algebraic over $k$.  
    \end{enumerate}
\end{proposition}

\begin{proof}
The forward direction follows from \cite[Proposition 5.1]{MultiDegreePositive}, which we briefly summarize here using the current notation. 
Let $f$ be a realizable volume polynomial over $k$ with support $J$.
Let $X$ be a subvariety of $\PP^\mu$ obtained from $f$ by the \emph{basic construction} in Section~\ref{sec:volumecovolume}.
For any $A \subseteq E$, we consider the coordinate projection
\[
\pi_A: \PP^\mu=\prod_{i \in E} \PP^{\mu_i} \longrightarrow \prod_{i \in A} \PP^{\mu_i}.
\]
Let $F$ be the function field of $X$, and let $F_i$ be the function field of the projection of $\pi_i(X)$. The function field of $\pi_A(X)$ is $F_A$. By \cite[Theorem 3.12]{MultiDegreePositive}, we know that the rank function of $\mathscr{P}$ satisfies  
\[
h(A)= \dim \pi_A(X)=\trdeg_k(F_A) \ \ \text{for all $A \subseteq E$}.
\]
This shows that $\mathscr{P}$ is algebraic over $k$.

For the other implication, suppose that there are field extensions $k \subseteq F_i \subseteq F$ such that
 \[
h(A)= \trdeg_k(F_A) \ \ \text{for all $A \subseteq E$.}
 \] 
We construct a subvariety $X$ of a product of projective spaces over $k$ such that the associated realizable volume polynomial has the support  $J$.
Since the composite of algebraic extensions is an algebraic extension, we may replace $F_i$ with a finitely generated subfield of the same transcendence degree over $k$ to assume without loss of generality that each $F_i$ is a finitely generated extension over $k$.
Choose a finite set of generators $f_{i,1},\ldots,f_{i,\mu_i}$ of each field extension $F_i/k$, and let $R$ be the $k$-subalgebra of $F$ generated by $f_{i,j}$ over $k$. The functions $f_{i,j}$ on the specturm of $R$ define a closed immersion
\[
 \Spec(R)
\;\xrightarrow{\quad f_{i,j}\quad}\;
\mathbb A^\mu
\;=\;
\prod_{i\in E}\mathbb A^{\mu_i}.
\]
Writing $X$ for the closure of the image of this map in the standard compactification $\PP^\mu$ of $\mathbb{A}^\mu$, 
\[
\dim \pi_A(X)=\dim \Spec\, k[f_{i,j}]_{i \in A}=\trdeg_k (F_A)=h(A) \ \ \text{for any $A \subseteq E$.}
\]
By  \cite[Theorem 3.12]{MultiDegreePositive},  the support of  the realizable volume polynomial associated to $X \subseteq \PP^\mu$ is equal to $J$. 
\end{proof}

In particular, a matroid $M$ is algebraic over $k$ if and only if there is a multiaffine realizable volume polynomial over $k$ whose terms with nonzero coefficients correspond to the bases of $M$.

\begin{example}\label{ex:Fano}
The Fano matroid $F_7$ is the matroid on seven elements whose bases are the three-element subsets that are not colinear in the following picture of the Fano plane:
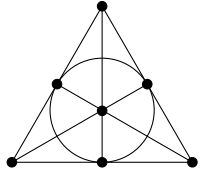
\begin{center}
\begin{tikzpicture}[scale=1.2]
\draw (0,0) -- (2,0);
\draw (2,0) -- (1,1.73);
\draw (1,1.73) -- (0,0);
\draw (0,0) -- (1.5,.866);
\draw (2,0) -- (.5,.866);
\draw (1,1.73) -- (1,0);
\draw (1,0.577) circle [radius=0.577];
\draw [fill=black] (0,0) circle (1.5pt); 
\draw [fill=black] (2,0) circle (1.5pt); 
\draw [fill=black] (1,1.73) circle (1.5pt); 
\draw [fill=black] (1.5,.866) circle (1.5pt); 
\draw [fill=black] (.5,.866) circle (1.5pt); 
\draw [fill=black] (1,0) circle (1.5pt); 
\draw [fill=black] (1,0.57) circle (1.5pt); 
\end{tikzpicture}
\end{center}
The basis generating polynomial $f$ of $F_7$ is a degree $3$ homogeneous polynomial in $7$ variables that has $28$ squarefree monomial summands.
We observe that $f$ is a realizable volume polynomial over $k$ if and only if the characteristic of $k$ is $2$. 
The ``if'' direction follows from the 
the general construction of arrangement Schubert varieties from linear realizations of matroids in \cite[Theorem 1.3]{Ardila-Boocher}. See \cite[Section 3]{HuhICM} for a brief overview.
The ``only if'' direction follows from Proposition~\ref{prop:algebraicPolymatroid}, together with 
Lindstr\"om's theorem  that  the Fano matroid is algebraic over $k$ only if the characteristic of $k$ is $2$  \cite{Lindstrom}.
It is not known whether $f$ is a volume polynomial over $k$ when the characteristic of $k$ is not $2$.
\end{example}

This connection to matroid theory yields several corollaries, which we present using the standard terminology and results in \cite{WelshMatroidTheory} and \cite{OxleyMatroidTheory}. The arguments presented here apply to polymatroids with little or no modification.

\begin{corollary}
If a matroid $M$ is algebraic over $k$, then every minor of $M$ is algebraic over $k$.
\end{corollary}

\begin{proof}
Let $f$ be a realizable volume polynomial over $k$ whose support is the set of bases of $M$, viewed as a set of squarefree monomials. Then the derivative $\partial_i f$ is a realizable volume polynomial over $k$ by Theorem~\ref{thm:covolumecharacterization}, and the support of $\partial_i f$ is the set of bases of the contraction $M/i$. Similarly,  the evaluation $f|_{x_i=0}$ is a realizable volume polynomial over $k$ by Proposition~\ref{prop:linmap}, and the support of $f|_{x_i=0}$ is the set of bases of  the deletion of $M \setminus i$.
The conclusion follows from Proposition~\ref{prop:algebraicPolymatroid}.
\end{proof}

This basic statement is typically deduced from a theorem of Lindstr\"om \cite{Lindstrom89}, who proved Piff's conjecture from \cite{Piff} that $M$ is algebraic over $k$ if $M$ is algebraic over $k(t)$. 
See \cite[Corollary 6.7.14]{OxleyMatroidTheory}, and compare \cite[Section 11.3]{WelshMatroidTheory}. 
The following result strengthens Lindstr\"om’s theorem.

\begin{corollary}
The following conditions are equivalent for a matroid $M$ and $p \ge 0$.
\begin{enumerate}[(1)]\itemsep 5pt
\item The matroid $M$ is algebraic over some field of characteristic $p$.
\item The matroid $M$ is algebraic over all fields of characteristic $p$.
\end{enumerate}
\end{corollary}

This recovers another result of Lindstr\"om \cite{Lindstrom88} that $M$ is algebraic over the prime field of $k$ if $M$ is algebraic over $k$.

\begin{proof}
 The statement follows from combining Propositions~\ref{prop:ArbitraryField0} and ~\ref{prop:algebraicPolymatroid}.
\end{proof}

Recall that the \emph{truncation} of a rank $d$ matroid $M$ is a matroid on the same ground set where a subset is a base if and only if it is an independent set in $M$ with exactly $d-1$ elements.
The next result is a strengthening of a theorem of Piff, who proved that the truncation of a matroid algebraic over $k$ is algebraic over some transcendental extension of $k$ \cite[Section 11.3, Theorem 2]{WelshMatroidTheory}.

\begin{corollary}\label{cor:truncation}
If $M$ is algebraic over $k$, then the truncation of $M$ is algebraic over $k$.
\end{corollary}

\begin{proof}
Since any nonnegative linear form is a realizable covolume polynomial,  $(\sum_i \partial_i) \circ f(x)$ is a realizable volume polynomial over $k$ for any realizable volume polynomial $f(x)$ over $k$ by Theorem~\ref{thm:covolumecharacterization}. 
If the support of $f$ is the set of bases of $M$, then the support of $(\sum_i \partial_i) \circ f(x)$ is the set of bases of the truncation of $M$. The conclusion follows from Proposition~\ref{prop:algebraicPolymatroid}.
\end{proof}

The \emph{dual} a matroid $M$ on $E$, denoted $M^*$, is the matroid on $E$ where a subset is a basis if and only if its complement is a basis of $M$.
We observe that the duality in matroid theory corresponds to the duality between volume and covolume polynomials. Using this connection, we prove the dual statement of Corollary~\ref{cor:truncation}. Recall that the \emph{Higgs lift} of $M$ is the dual of the truncation of $M^*$.

\begin{corollary}\label{cor:lift}
If $M$ is algebraic over $k$, then the Higgs lift of $M$ is algebraic over $k$.
\end{corollary}

\begin{proof}
Since any nonnegative linear form is a realizable volume polynomial,  $(\sum_i x_i) \cdot g(\partial)$ is a realizable covolume polynomial over $k$ for any realizable covolume polynomial $g(\partial)$ over $k$ by Theorem~\ref{thm:volumecharacterization}. The conclusion follows from Proposition~\ref{prop:algebraicPolymatroid}.
\end{proof}

We now exploit multiplicative properties of realizable volume polynomials and realizable covolume polynomials to prove a result in matroid theory. 

\begin{definition}
Let $M_1$ and $M_2$ be matroids on a common ground set~$E$.
\begin{enumerate}[(1)]\itemsep 5pt
\item The \emph{union} $M_1 \vee M_2$ is the matroid on $E$ whose independent sets are the subsets of $E$ that can be expressed as a union $I_1 \cup I_2$, where $I_1$ is an independent set of $M_1$ and $I_2$ is an independent set of $M_2$.
\item The \emph{intersection} $M_1 \wedge M_2$ is the matroid on $E$ whose spanning sets are the subsets of $E$ that can be expressed as an intersection $S_1 \cap S_2$, where $S_1$ is a spanning set of $M_1$ and $S_2$ is a spanning set of $M_2$.
\end{enumerate}
The two constructions are related by the matroid duality by $M_1 \wedge M_2 =(M_1^* \vee M_2^*)^*$.
\end{definition}

We say that $M_1 \vee M_2$ has the \emph{expected rank} if the set of bases of $M_1\vee M_2$ is 
\[
\big\{B=B_1\cup B_2\mid B_1\in J_1, B_2\in J_2, \text{ and }B_1\cap B_2=\varnothing\big\},
\]
where $J_1$ is the set of bases of $M_1$ and  $J_2$ is the set of bases of $M_2$. Dually, we say that 
$M_1 \wedge M_2$ has the expected rank if the set of bases of $M_1 \wedge M_2$ is 
\[
\big\{B=B_1\cap B_2\mid B_1\in J_1, B_2\in J_2, \text{ and }B_1\cup B_2=E\big\}.
\]

Welsh's theorem states that the union $M_1 \vee M_2$ is algebraic over $k$ when both $M_1$ and $M_2$ are algebraic over $k$ \cite[Section~11.3, Theorem~4]{WelshMatroidTheory}. We prove the dual statement.

\begin{theorem}\label{thm:MatroidIntersection}
The intersection $M_1\wedge M_2$ is algebraic over $k$
 when both $M_1$ and $M_2$ are algebraic over $k$.
\end{theorem}

The result is consistent with a positive answer to one of the oldest questions in the subject: Is the dual of an algebraic matroid also algebraic?

\begin{proof}
 Let $J_i$ be the set of bases of $M_i$ and let $K_i$ be the set of bases of $N_i\coloneq M_i^*$. 
  By Proposition \ref{prop:algebraicPolymatroid}, there are realizable volume polynomials $f_1$ and $f_2$ over $k$ such that the support of $f_1$ is $J_1$ and the support of $f_2$ is $J_2$.
 Thus, there are realizable covolume polynomials $g_1$ and $g_2$ over $k$ such that the support of $g_1$ is $K_1$ and the support of $g_2$ is $K_2$.
By Theorem~\ref{thm:covolumeproduct}, the product of realizable covolume polynomials over $k$ is a realizable covolume polynomial over $k$, so $g_1g_2$ is a realizable covolume polynomial over $k$.
If $M_1 \wedge M_2$ has the expected rank, the set of bases of $M_1 \wedge M_2$ is the support of the realizable volume polynomial  $(g_1g_2)\circ x^E$. By Proposition~\ref{prop:algebraicPolymatroid}, this implies that $M_1 \wedge M_2$ is algebraic over $k$.

When $M_1 \wedge M_2$ fails to have the expected rank, we use 
a result of Cunningham \cite[Theorem 2]{Cunningham}: We can repeatedly truncate $N_1$ and $N_2$ without changing their union so that the union of the truncations has the expected rank. In other words, we can take Higgs lifts of $M_1$ and $M_2$ without changing their intersection so that the intersection of the Higgs lifts has the expected rank. 
We only need to check that the  Higgs lift of a matroid algebraic over $k$ is algebraic over $k$, which is Corollary~\ref{cor:lift}.
\end{proof}

\begin{remark}
In the proof of Theorem~\ref{thm:MatroidIntersection}, the use of Theorem~\ref{thm:covolumeproduct} cannot be replaced by Corollary~\ref{cor:covolumeproduct}, which states that the product of covolume polynomials over $k$ is a covolume polynomial over $k$.
\end{remark}

\begin{remark}
Welsh's theorem on the union of algebraic matroids  can be proved using the same argument. In place of Theorem~\ref{thm:covolumeproduct}, one uses Corollary~\ref{cor:prodofvol}, the easier fact that the product of realizable volume polynomials over $k$ is a realizable volume polynomial over $k$.
\end{remark}

\begin{remark}
A real $(1,1)$-class $[\omega]$ on a compact K\"ahler manifold $Y$ is \emph{semipositive}\footnote{When $Y$ is a smooth projective variety over $\CC$, every semiample divisor class is semipositive, but not every nef divisor class is semipositive \cite{yau1974curvature}. For an example of a nef divisor class that is semipositive but not semiample, see \cite{koike2015minimal}. A comprehensive survey of semipositive classes on compact K\"ahler manifolds can be found in \cite{Tosatti}.}  if it contains a smooth semipositive representative, that is, if there is a smooth function $\varphi$ on $Y$ such that
\[
\omega+i\partial \overline{\partial} \varphi \ge 0.
\]
A degree $d$ homogeneous polynomial $f$  is a \emph{realizable analytic volume polynomial} if there is a $d$-dimensional compact K\"ahler manifold $Y$ and semipositive classes $[\omega_1],\dots,[\omega_n]$ such that 
     \[
     f(x_1,\dots,x_n)=\int_Y (x_1\omega_1+\dots +x_n\omega_n)^{\wedge d}.
     \]
    A homogeneous polynomial $f$ is an \emph{analytic volume polynomial} if it is a limit of realizable analytic volume polynomials.
The mixed Hodge--Riemann relations for compact K\"ahler manifolds imply that an analytic volume polynomial is a Lorentzian polynomial, so the support of an analytic volume polynomial defines a polymatroid. An \emph{analytic polymatroid} is the support of 
a realizable analytic volume polynomial. 
These generalize algebraic polymatroids over $\mathbb{C}$ in the same way K\"ahler manifolds generalize smooth projective varieties over $\mathbb{C}$.
It will be interesting to decide whether analytic volume polynomials satisfy the analogues of Theorems~\ref{thm:covolumecharacterization},~\ref{thm:volumecharacterization},~\ref{thm:symbol}, and to decide whether 
the classes of analytic matroids and algebraic matroids over $\CC$ coincide. 
\begin{comment}
\begin{proposition}
    The support of an analytic volume polynomial $f$ is a polymatroid.
\end{proposition}

\begin{proof}
    By \cite{BrandenHuh} it is enough to prove that $f$ is Lorentzian. As the locus of Lorentzian polynomials is closed, we may assume that $\omega_1,\dots,\omega_n$ are K\"ahler classes and $f(x)=\int_Y (x_1\omega_1+\dots +x_n\omega_n)^{\wedge d}$. Let $\omega_y=\sum_{j=1}^n y_j\omega_j$ for arbitrary nonnegative $y_j$. Gromov's version of Hodge Index Theorem for K\"ahler manifolds implies:
    \[\left(\int_Y \omega_{y^1}\wedge\dots \wedge \omega_{y^{d-2}}\wedge w_{y'}\wedge\omega_{y''}\right)^2\geq \left(\int_Y  \omega_{y^1}\wedge\dots \wedge \omega_{y^{d-2}}\wedge w_{y'}^{\wedge 2}\right)\left(\int_Y  \omega_{y^1}\wedge \dots \wedge \omega_{y^{d-2}}\wedge \omega_{y''}^{\wedge 2}\right).\]
    This means that differentiating $f$ with respect to any $d-2$ directions from the positive orthant results in a Lorentzian quadric. This implies that $f$ is Lorentzian. 
\end{proof}
\end{comment}
\end{remark}

%\bibliography{volpolpres}
\bibliographystyle{amsalpha}

\end{document}